\documentclass[a4,12pt, leqno]{amsart}
\usepackage{amsmath, amsthm, amscd, amsfonts, amssymb, graphicx, color}

\def\1{\mathbbmss{1}}
\def\X{\mathcal{X}}
\def\Y{\mathcal{Y}}
\def\A{\mathcal{A}}
\def\N{\mathcal{N}}
\def\L{\mathcal{L}}
\def\R{\mathcal{R}}
\def\H{\mathcal{H}}

\def\sfstp{{\hskip-1em}{\bf.}{\hskip.5em}}
\newtheorem{theorem}{Theorem}[section]
\newtheorem{lemma}[theorem]{Lemma}
\newtheorem{proposition}[theorem]{Proposition}
\newtheorem{corollary}[theorem]{Corollary}
\theoremstyle{definition}
\newtheorem{definition}[theorem]{Definition}

\theoremstyle{remark}
\newtheorem{remark}[theorem]{Remark}
\numberwithin{equation}{section}

\begin{document}
\setcounter{page}{1}

\title[CO-EP BANACH ALGEBRA ELEMENTS]{CO-EP BANACH ALGEBRA ELEMENTS}

\author[J. BEN\'ITEZ, E. BOASSO, V. RAKO\v CEVI\'C]{JULIO BEN\' ITEZ, ENRICO BOASSO AND VLADIMIR RAKO\v CEVI\'C}



\begin{abstract}In this work, given a unital Banach algebra 
$\A$ and  $a\in \A$ such that
$a$ has a Moore-Penrose inverse $a^\dagger$, it will be characterized when $aa^\dagger-
a^\dagger a$ is invertible. A particular subset of this class of objects will also be
studied. In addition, perturbations of this class of elements will be  studied.
Finally, the Banach space operator case will be also considered.

\end{abstract} \maketitle

\section{Introduction and preliminaries}

A square matrix $A$ is said to be EP if 
$\N(A)=\N(A^*)$. This notion was introduced
in \cite{S} and since then many authors have studied EP matrices.  Since a necessary and sufficient condition  for a matrix $A$ to be EP is the fact that 
$A$ commutes with its Moore-Penrose inverse, the notion under consideration
has been extended to Hilbert space operators and
$C^*$-algebra elements, see for example  \cite{CM, Br, HM1, HM2, K1, DD, DK, Ben}.\par

\indent In the context of Banach algebras the notion of Moore-Penrose inverse was
introduced by V. Rako\v cevi\' c in \cite{R1} and its basic
properties were studied in \cite{R1, R2, M, Bo}.
In the recent past EP Banach space operators and EP Banach algebra elements,
i.e., Moore-Penrose invertible operators or elements of
a Banach algebra such that they commute with their Moore-Penrose
inverse, were introduced and characterized, see \cite{Bo, Bo2, Bo1, MD1}. \par

\indent The main object of this work is to study a complementary class of object, the
co-EP Banach algebra elements, i.e., the Moore-Penrose invertible elements $a\in \A$ such that 
$aa^\dagger-a^\dagger  a$ is nonsingular, where $\A$ is a Banach algebra, $a\in \A$ and $a^\dagger$ denotes the Moore-Penrose inverse of $a$.
This class of objects were studied for matrices (\cite{BR1, BR3}), for $C^*$-algebras (\cite{BR2}) and for  rings (\cite{BLR}). In section 2, 
co-EP Banach algebra elements  will be characterized. In addition, it will be also given necessary
and sufficient conditions that ensure the nonsigularity of $aa^\dagger+a^\dagger a$, $a\in \A$, $\A$ a unital Banach algebra.
Moreover, co-EP Banach space operators defined on finite and infinite dimensional Banach space will be also studied. 
On the other hand, in section 3 a particular set of co-EP Banach algebra elements will be characterized and in section 4 perturbations of
co-EP elements will be considered.

\indent It is worth noticing that since $aa^\dagger$ and $a^\dagger  a$ are idempotents, characterizing co-EP elements is related to the following problem: given 
two idempotents $p$ and $q$,  find conditions that ensure
the nonsingularity of $p-q$. This problem was studied in different frames and by many authors, see for example \cite{V2, Pt, Bu1, Bu2, KR2, KR, KR5, KRS}. 

\indent It is also important to observe that due to the lack of an involution on
a Banach algebra, and in particular on the Banach algebra of
bounded and linear maps defined on a Banach space, the results presented 
in this work give a new
insight into the cases where the involution does exist. In fact,
the results and proofs presented do not depend on a particular norm, the Euclidean norn,
but they hold for any norm. In particular, the results considered in this work also apply
to co-EP matrices defined using an arbitrary norm on a
finite dimensional vector space.\par

\indent From now on  $\X$ will denote a complex Banach space and $\L(\X, \Y)$ the
Banach algebra of all bounded and linear maps defined on $\X$ with values
in the Banach space $\Y$. As usual, when $\X=\Y$, $\L(\X,\Y)$ will be
denoted by $\L(\X)$. Let $I$ denote the identity map of $\L(\X)$.
In addition, if $T\in \L(\X,\Y)$, then $\N(T)\subseteq \X$ and
$\R(T)\subseteq \Y$ will stand for the null space and the
range of $T$ respectively. \par

\indent On the other hand, $\A$ will denote a unital complex Banach algebra with unit $1$
and $\A^{-1}$ will stand for the set of all invertible elements of $\A$.
If $a\in \A$, then $L_a \colon \A\to \A$
and $R_a\colon \A\to \A$ will denote the map defined by
left and right multiplication respectively:
$$
L_a(x)=ax, \hskip2truecm R_a(x)=xa, \hskip2truecm (x\in \A).
$$
Moreover, the following notation will be used:
$a^{-1}(0)=\N(L_a)$, $a\A=\R(L_a)$, $a_{-1}(0)=\N(R_a)$ and
$\A a= \R(R_a)$. Observe that $a\in \A^{-1}$ if and only if $L_a\in \L(\A)^{-1}$.\par

\indent Recall that an element $a\in \A$ is called \it{regular}, \rm
if it has a \it{generalized inverse}, \rm namely if there exists $b\in \A$ such that
$$
a=aba.
$$

\noindent Furthermore, a generalized inverse $b$ of a regular
element $a\in \A$ will be called \it{normalized}, \rm if $b$ is regular
and $a$ is a generalized inverse of $b$, equivalently,
$$
a=aba, \hskip2truecm b=bab.
$$

\noindent Note that if $b$ is a generalized inverse of $a$,
then $bab$ is a normalized generalized inverse
of $a$. What is more, when $b\in \A$ is a normalized generalized inverse
of $a\in \A$, $ab$ and $ba$ are idempotents and the following identities hold:\par
\renewcommand\arraystretch{1.5}
\begin{equation*} 
\begin{array}{ccccccc}
ab\A=a\A, & \hspace{.5cm} & (ab)^{-1}(0)=b^{-1}(0), & \hspace{.5cm} & \A ab=\A b,
& \hspace{.5cm} & (ab)_{-1}(0)=a_{-1}(0), \\
ba\A=b\A, & \hspace{.5cm} & (ba)^{-1}(0)=a^{-1}(0), & \hspace{.5cm} & \A ba=\A a,
& \hspace{.5cm} & (ba)_{-1}(0)=b_{-1}(0).
\end{array}
\end{equation*}

\noindent Recall also that
\begin{equation*} 
\begin{array}{ccc}
(1-ab)\A=(ab)^{-1}(0), & \hspace{.5cm} & \A(1-ab)=(ab)_{-1}(0), \\
(1-ba)\A=(ba)^{-1}(0), & \hspace{.5cm} & \A(1-ba)=(ba)_{-1}(0).
\end{array}
\end{equation*}

\indent Next follows the key notion in the definition of the
Moore-Penrose inverse in the context of Banach algebras.\par

\begin{definition}\label{Def1}
Given a unital Banach algebra $\A$, an element $a\in \A$ is said to be hermitian,
if $\parallel \exp(ita)\parallel =1$,  for all $ t\in \mathbb{R}$.
\end{definition}

\indent As regard equivalent definitions and the main properties of hermitian Banach
algebra elements and Banach space operators, see \cite{BD, Do, Lu, Pa, V}.
In the conditions of the previous
definition, recall that if $\A$ is a $C^*$-algebra, then
$a\in \A$ is hermitian if and only if $a$ is self-adjoint,
see \cite[Proposition 20, Section 12, Chapter I]{BD}.
Furthermore, $\H=\{a\in \A\colon \hbox{  }a \hbox{ is
hermitian}\}\subseteq \A$ is a closed linear vector space over the real field
(\cite{V}, \cite[Theorem 4.4, Chapter 4]{Do}). Since $\A$ is unital, $1 \in \H$,
which implies that $a\in \H$ if and only if $1 - a\in \H$.\par

\indent Now the definition of Moore-Penrose invertible Banach algebra elements
will be recalled.\par

\begin{definition} \label{def2}Let $\A$ be a unital Banach algebra and consider $a\in \A$. If there exists
$x\in \A$ such that $x$ is a normalized generalized inverse of $a$ satisfying that
$ax$ and $xa$ are hermitian, then $x$ will be called the {\em Moore-Penrose inverse} of $a$,
and it will be
denoted by $a^{\dag}$.
\end{definition}

\indent Recall that according to \cite[Lemma 2.1]{R1}, given $a\in \A$, there is at most
one $x\in \A$ satifying the conditions of Definition \ref{def2}. In addition, recall that
even for matrices with an arbitrary norm, the Moore-Penrose inverse could not exist
(see \cite[Remark 4]{Bo}).
Let $\A^{\dag}$ denote the set of all
Moore-Penrose invertible elements of $\A$. As regard the Moore-Penrose inverse in Banach algebra,
see \cite{Bo, M, R1, R2}. For the original definition of the Moore-Penrose
inverse for matrices, see \cite{Pe}.\par

\indent Given $a\in \A$ such that $a^{\dag}$ exists, $a\in \A$ is said to be {\em EP},
if $aa^{\dag}=a^{\dag}a$. EP elements in Banach algebras have been considered in
\cite{Bo, Bo2, Bo1, BDM, MD1}. Next the objects that will be studied in this article will be introduced.\par

\begin{definition} \label{def3} Let $\A$ be a unital Banach algebra and $a\in \A^{\dag}$.
The element $a$ will be said to be {\em co-EP}, if $aa^{\dag}-a^{\dag}a\in \A^{-1}$.
\end{definition}

\indent Given a unital Banach algebra $\A$, the set of all co-EP elements
of $\A$ will be denoted by $\A_{co}^{EP}$. Co-EP elements have been studied in \cite{BR1, BR2, BR3}.\par

\section {\sfstp Co-EP elements}

\noindent In this section co-EP Banach algebra elements and Banach space operators
will be characterized. Co-EP matrices were studied in \cite{BR1, BR3} and co-EP $C^*$-algebra elements
in \cite{BR2}. However, in order to characterize co-EP elements, some preparation is needed.\par


\begin{proposition}\label{pro8}
Let $\A$ be a unital Banach algebra, $a\in \A^{\dag}$, and $\lambda, \mu \in \mathbb{C}
\setminus \{0\}$. The following statements hold.
\begin{itemize}
\item[{\rm (i)}] $L_{\lambda a + \mu a^{\dag}} ((1-aa^{\dag})\A ) =
    L_a ((1-aa^{\dag}) \A )\subseteq aa^{\dag}\A$.
\item[{\rm (ii)}] $L_{\lambda a + \mu a^{\dag}} ((1-a^{\dag}a)\A ) =
    L_{a^{\dag}} ((1-a^{\dag}a)\A )\subseteq a^{\dag}a\A$.
\item[{\rm (iii)}] The following statements are equivalent:
    \begin{itemize}
    \item[{\rm (iii.a)}] $(1-aa^{\dag})\A\cap (1-a^{\dag}a)\A=0$;
    \item[{\rm (iii.b)}]
        $\N(L_{\lambda a + \mu a^{\dag}}) \cap (1 - aa^\dag ) \A =0$;
    \item[{\rm (iii.c)}]
        $\N(L_{\lambda a + \mu a^{\dag}}) \cap (1 - a^\dag a) \A = 0$.
    \end{itemize}
\item[{\rm (iv)}] If $a^{\dag} \in (1-aa^{\dag})\A+  (1-a^{\dag}a)\A$,
    then $L_{\lambda a + \mu a^{\dag}}((1 - aa^\dag)\A ) =
    L_a((1 - aa^\dag)\A) = aa^{\dag}\A$.\par
\item[{\rm (v)}] If $a \in (1-aa^{\dag})\A + (1-a^{\dag}a)\A$,
    then $L_{\lambda a + \mu a^{\dag}}(( 1 - a^\dag a) \A ) =
    L_{a^{\dag}}(( 1 - a^\dag a) \A ) = a^{\dag}a\A$.
\end{itemize}
\end{proposition}

\begin{proof}
(i). Let $x\in \A$. The equalities 
$(\lambda a + \mu a^{\dag})(1 - aa^{\dag})x =
\lambda a(1 - aa^{\dag})x$ and 
$a(1 - aa^{\dag})x = 
aa^\dag a(1 - aa^{\dag})x$ prove (i).

(ii). Apply an argument similar to the one used in the proof of statement (i).\par

(iii.a) $\Rightarrow$ (iii.b) follows from $(1 - aa^\dag)\A = (aa^\dag)^{-1}(0) =
(a^\dag)^{-1}(0)$ and $(1 -  a^\dag a)\A = (a^\dag a)^{-1}(0) = a^{-1}(0)$.

(iii.b) $\Rightarrow$ (iii.a). Let $x \in (1-aa^{\dag}) \A\cap (1-a^{\dag}a)\A$.
Since $x \in a^{-1}(0) \cap (a^\dag)^{-1}(0)$,
$x \in \N(L_{\lambda a + \mu a^{\dag}}) \cap (1 - aa^\dag ) \A =0$.

The equivalence (iii.a) $\Leftrightarrow$ (iii.c) can be proved in a similar way.

(iv). Let $c$, $d\in \A$ such that  $a^{\dag}=(1-aa^{\dag})c+(1-a^{\dag}a)d$.
In particular, if $x\in \A$, then $aa^{\dag}x=a(1-aa^{\dag})cx
= (\lambda a + \mu a^\dag)(1-aa^\dag )c (\lambda^{-1} x)$.\par

(v). Apply an argument similar to the one used in the proof of statement (iv).
\end{proof}

In the following theorem, co-EP Banach algebra elements will be characterized.
Note that this theorem is stronger than \cite[Theorem 2]{BLR}.\par

\begin{theorem}\label{thm7}
Let $\A$ be a unital Banach algebra, $a\in \A^{\dag}$ and $\lambda, \mu \in \mathbb{C}
\setminus \{0 \}$.
The following statements are equivalent.
\begin{itemize}
\item[{\rm (i)}] $aa^{\dag}-a^{\dag}a\in \A^{-1}$;
\item[{\rm (ii)}] $\A=a\A\oplus  a^{\dag}\A$ and  $\A=\A a\oplus \A a^{\dag}$;

\item[{\rm (iii)}] $\lambda a + \mu a^{\dag}\in \A^{-1}$ and $a\A\cap a^{\dag}\A=0$;
\item[{\rm (iv)}] $\lambda a + \mu a^{\dag}\in \A^{-1}$ and exists
	and idempotent $h$ such that $ha=a$, $ha^\dag = 0$;
\item[{\rm (v)}] $L_{\lambda a + \mu a^{\dag}}$, 
	$R_{\lambda a + \mu a^{\dag}}\in \L(\A )$ are right invertible and
    $a\A\cap a^{\dag}\A =0$;
\item[{\rm (vi)}] $aa^{\dag}+a^{\dag}a\in \A^{-1}$ and $a\A  \cap  a^{\dag}\A=0$;
\item[{\rm (vii)}] $\lambda a + \mu a^{\dag}\in \A^{-1}$ and exists
	and idempotent $k$ such that $ak=a$, $a^\dag k= 0$;
\item[{\rm (viii)}] $aa^{\dag}+a^{\dag}a\in \A^{-1}$ and $\A a\cap \A a^{\dag}=0$;
\item[{\rm (ix)}] $\lambda a + \mu a^{\dag}\in \A^{-1}$ and $\A a\cap \A a^{\dag}=0$.

\end{itemize}
\end{theorem}

\begin{proof} (i) $\Leftrightarrow$ (ii). According to \cite[Theorem 3.2]{KR}, statement (i)
is equivalent to $\A=aa^{\dag}\A\oplus a^{\dag}a\A= a\A\oplus a^{\dag}\A$ and $\A=\A aa^{\dag}\oplus \A a^{\dag}a=
\A a^{\dag}\oplus \A a$.\par

\noindent (ii) $\Rightarrow$ (iii).  According to \cite[Lemma 2.1]{KR},
$\A = (1-aa^{\dag})\A\oplus (1-a^{\dag}a)\A$. 
Since $A=aa^{\dag}\A\oplus a^{\dag}a\A$, according to the proof of 
Proposition \ref{pro8}(i)-(ii),
$$
L_{\lambda a+\mu a^{\dag}}\colon (1-aa^{\dag})\A\oplus (1-a^{\dag}a)\A\to aa^{\dag}\A\oplus a^{\dag}a\A,
$$
has the following matricial form:
$\begin{pmatrix}                                
L_{\lambda a} &0\\
0& L_{\mu a^{\dag}}\\
\end{pmatrix}$.
Therefore, according to  Proposition \ref{pro8}(iii)-(v), $L_{\lambda a + \mu a^\dag}$ is 
invertible, equivalently, $\lambda a + \mu a^{\dag}\in A^{-1}$.
The remaining identity is  clear.\par


\noindent (iii) $\Rightarrow$ (iv).
Let $x = (\lambda a + \mu a^{\dag})^{-1}$. Since
$1 = a(\lambda x) + a^\dag (\mu  x)$, $\A=a\A+a^{\dag}\A$ and by hypothesis, $\A=a\A\oplus a^{\dag}\A$.
Now, (iv) follows from \cite[Lemma 2.1]{KR}. \par

\noindent (iv) $\Rightarrow$ (ii).
As in the previous implication, the condition 
$\lambda a + \mu a^\dagger \in \A^{-1}$ implies $\A = a\A + a^\dag \A = \A a + \A a^\dagger$.
Consider any $x \in a\A \cap a^\dag A$. There exist $u,v \in \A$ such that
$x=au = a^\dag v$, which leads to $hau=h a^\dag v$, hence
$au=0$, and thus, $x=0$. It has been proved that $a\A \cap a^\dag A = 0$. 
Now consider $y \in \A a \cap \A a^\dagger$, i.e., 
$y=wa = ta^\dagger$ for some $w,t \in \A$.
Since $\lambda a + \mu a^\dag \in \A^{-1}$, exists $z \in \A$ such that
$(\lambda a + \mu a^\dag)z= 1$ (in fact, $z$ is the standard inverse
of $\lambda a + \mu a^\dagger$). Now, 
$h=h(\lambda a + \mu a^\dag)z = \lambda az$, and thus $\mu a^\dag z = 1-h$.
Since $yz =  waz=ta^\dag z$, $yz = \lambda^{-1}wh = \mu^{-1}t(1-h)$, which
by postmultiplying by $h$ leads to $yz=0$. 
By using $z \in \A^{-1}$, $y=0$.   As a result,
$\A a \cap \A a^\dag = 0$.

\noindent (iii) $\Rightarrow$ (v). Clear.\par

\noindent (v) $\Rightarrow$ (iii).
If $L_{\lambda a + \mu a^{\dag}}$ and
$R_{\lambda a + \mu a^{\dag}}\in \L(\A)$ are right invertible,
then ${\lambda a + \mu a^{\dag}}\in \A^{-1}$.\par

\noindent (ii) $\Rightarrow$ (vi). Recall that $\A = (1-aa^{\dag})\A \oplus
(1-a^{\dag}a)\A$ (\cite[Lemma 2.1]{KR}) and
consider $L_{aa^{\dag}+a^{\dag}a}\colon \A\to \A$. Let $x\in \A$ such that
$(aa^{\dag}+a^{\dag}a)x=0$. As a result, $aa^{\dag}x\in a\A\cap a^{\dag}\A=0$. Similarly, $a^{\dag}ax=0$.
Consequently, $x\in  (aa^{\dag})^{-1}(0)\cap (a^{\dag}a)^{-1}(0) =
(1-aa^{\dag})\A\cap (1-a^{\dag}a)\A=0$, i.e.,
$\N(L_{aa^{\dag}+a^{\dag}a})=0$.
On the other hand, given $x\in \A$, there exist $y$, $z\in \A$ such that
$x=(1-aa^{\dag})y + (1-a^{\dag}a)z$.
Then,
$$
aa^{\dag}x = aa^{\dag}(1-aa^{\dag})y+aa^{\dag}(1-a^{\dag}a)z
= aa^{\dag}(1-a^{\dag}a)z = (aa^{\dag}+a^{\dag}a)(1-a^{\dag}a)z.
$$
As a result, $aa^{\dag}\A\subseteq \R(L_{aa^{\dag}+a^{\dag}a})$. Similarly,
$$
a^{\dag}ax = a^{\dag}a(1-aa^{\dag})y +a^{\dag}a(1-a^{\dag}a)z = 
a^{\dag}a(1-aa^{\dag})y = ( aa^{\dag}+a^{\dag}a ) (1-aa^{\dag})y,
$$
which implies that $a^{\dag}a\A\subseteq \R(L_{aa^{\dag}+a^{\dag}a})$. 
Since $A=aa^{\dag}\A\oplus a^{\dag}a\A$, 
$\R(L_{aa^{\dag}+a^{\dag}a})=\A$. 
Therefore, ${aa^{\dag}+a^{\dag}a}\in \A^{-1}$.

\noindent (vi) $\Rightarrow$ (ii). 
If $aa^{\dag}+a^{\dag}a\in \A^{-1}$, then  $\A=a\A+a^{\dag}\A$ and $\A=\A a+\A a^{\dag}$.
By hypothesis, $\A = a\A \oplus a^\dag A$. Now, \cite[Lemma 2.1]{KR} assures  
the existence of an idempotent $h$ such that $ha=a$ and $ha^\dag = 0$. To prove (ii), 
it is sufficient to prove $\A a \cap \A a^\dag = 0$ (this is rather similar to the 
proof of (iv) $\Rightarrow$ (ii)). If $y \in \A a \cap \A a^\dag $, there exist
$w,t \in \A$ such that $y=wa=ta^\dagger$. Let $z=(aa^\dag  + a^\dag a)^{-1}$.
Since $h = h(aa^\dag z + a^\dag a z) = aa^\dag z$ and $(aa^\dag + a^\dag a)z= 1$,
 $a^\dag a z = 1-h$. Therefore, $a^\dag h = a^\dag z$ and
$az = a(1-h)$. Since $y=wa=ta^\dagger$,  $yz=waz=ta^\dag z$, and thus
$yz=wa(1-h) = ta^\dag h$. Thus, $yz=0$ and using the invertibility of $z$, $y=0$.\par

\noindent (i) $\Leftrightarrow$ (vi), (vii), (viii) or (ix).
Condition (i) is invariant under the reversal of the algebra multiplication.
Precisely speaking, apply each of these conditions to the algebra
$(\A, \circ)$, where $x \circ y = yx$, and then reinterpret the results in the original algebra.
Conditions (iii), (iv), (vi) and (ii) yield that (i) is equivalent to
any of the conditions (vi), (vii),  (viii) or (ix) respectively.
\end{proof}


\indent The next result deals with a weaker condition than the invertibility of 
$aa^\dag + a^\dag a$ when $a \in \A^\dag$. 
It is noteworthy that $aa^\dag - a^\dag a \in \A^{-1}$ implies 
$aa^\dag + a^\dag a \in \A^{-1}$ (see Theorem \ref{thm7} or \cite[Theorem 3.5]{KR2}). 
Also, in \cite[Theorem 3.3]{KR2} the authors characterized the invertibility of 
$aa^\dag + a^\dag a$.
Furthermore, 
for $b \in \A$, observe that $L_b$ is injective if and only if 
the condition $bx=0$ implies $x=0$ when $x \in \A$ (i.e., $b$
satisfies a kind of left cancellation property). \par

\begin{proposition}\label{pro37}
Let $\A$ be a unital Banach algebra and consider $a\in \A^{\dag}$.
The following statements are equivalent.
\begin{itemize}
\item[{\rm (i)}] $L_{aa^{\dag}+a^{\dag}a}$ is injective.
\item[{\rm (ii)}] $a^\dag a \A \cap aa^\dag (1 - a^\dag a) \A= 0$ and
	$a^{-1}(0) \cap (a^\dag)^{-1}(0) = 0$.	
\end{itemize}
\end{proposition}
\begin{proof}
(i) $\Rightarrow$ (ii).
Let $x \in a^\dag a \A \cap aa^\dag (1 - a^\dag a) \A$.
There exist $u,v \in \A$ such that 
\begin{equation}\label{n1}
x = a^\dag u = aa^\dag (1 - a^\dag a)v.
\end{equation}
Now, 
$$
(aa^\dag + a^\dag a)x = aa^\dag x + a^\dag a x = 
aa^\dag aa^\dag (1 - a^\dag a)v + a^\dag aa^\dag u = 
a a^\dag (1 - a^\dag a)v + a^\dag u = 2x.
$$
Hence
$$
2(aa^\dag + a^\dag a)(1 - a^\dag a)v = 2 aa^\dag (1 - a^\dag a)v = 2x
= (aa^\dag + a^\dag a)x.
$$
The hypothesis implies $2(1 - a^\dag a)v = x$, and thus
$2aa^\dag (1 - a^\dag a)v = aa^\dag x$. By using \eqref{n1},
$2x=x$, hence $x=0$.

Let $y \in a^{-1}(0) \cap (a^\dag)^{-1}(0)$, i.e., 
$ay=a^\dag y=0$. It is evident that $(aa^\dag + a^\dag a)y = 0$, and
the hypothesis leads to $y=0$.

\noindent (ii) $\Rightarrow$ (i).
It is evident that $L_{aa^\dag + a^\dag a}(( 1 - a a^\dag ) \mathcal{A}) 
\subseteq a^\dag \mathcal{A}$ and 
$L_{aa^\dag + a^\dag a}(( 1 - a^\dag a) \mathcal{A}) 
\subseteq a \mathcal{A}$. Thus, if $L_1$ and $L_2$ are the restrictions
of $L_{aa^\dag + a^\dag a}$ to $(1-aa^\dag)\A$ and 
$(1-a^\dag a)\A$ respectively, then
$$
L_1: (1-aa^\dag)\A \to a^\dag \A
\qquad
\text{and}
\qquad
L_2: (1-a^\dag a)\A \to a \A.
$$
Let $x \in \N (L_{aa^\dag + a^\dag a})$. 
Since $x = x-aa^\dag x - a^\dag a x = (1 - aa^\dag) x + 
(1 - a^\dag a)x - x$, if 
$u=(1 - aa^\dag) x$ and $v=(1 - a^\dag a)x$, then
$2x=u+v$, and thus $0=L_1(u)+L_2(v)$.
Moreover,
$$
L_2(v)= (aa^\dag + a^\dag a) (1-a^\dag a)x
= 
a a^\dag (1-a^\dag a)x \in a a^\dag (1-a^\dag a) \A 
$$
and
$$
L_2(v) = -L_1(u) \in a^\dag \A = a^\dag a \A.
$$
By hypothesis, $L_2(v)=0$, and thus $L_1(u)=0$. Therefore,
$$
0 = (aa^\dag + a^\dag a)v = 
(aa^\dag + a^\dag a)(1-a^\dag a)x =
aa^\dag (1-a^\dag a)x = aa^\dag v, 
$$
which by a premultiplication by $a^\dagger$ leads to $0=a^\dag v $.
In other words, $v \in (a^\dag)^{-1}(0)$. It is evident that
$av = a(1-a^\dag a)x = 0$. Thus, $v \in a^{-1}(0) \cap (a^\dag)^{-1}(0) = 0$.
In a similar way it is possible to prove that  $u=0$. Since $2x=u+v$, $x=0$.
\end{proof}


In the following proposition the surjectivity of the operator $L_{aa^{\dag}+a^{\dag}a}\colon \A\to \A$ will be characterized.
\begin{proposition}\label{pro38}
Let $\A$ be a unital Banach algebra and consider $a \in \A^\dagger$.
Then, the operator $L_{aa^\dag + a^\dag a}\in \L(\A)$ is surjective if and only if 
$aa^\dag + a^\dag a$ is regular and $(aa^\dag + a^\dag a)_{-1}(0) = 0$.
\end{proposition}
\begin{proof}
Assume that $L_{aa^\dag + a^\dag a}$ is surjective. 
Since $1 \in L_{aa^\dag + a^\dag a}(\A)$, there exists
$u \in \A$ such that $(aa^\dag + a^\dag a)u=1$.
This, obviously, implies that $aa^\dag + a^\dag a$ is regular. Also, 
if $x \in \A$ satisfies $0=x(aa^\dag + a^\dag x)$, then by a postmultiplication by $u$, 
$0=x$.\par

To prove the converse, consider $b \in \A$ such that 
$(aa^\dag + a^\dag)b(aa^\dag + a^\dag a)=aa^\dag + a^\dag a$.
Since $[(aa^\dag + a^\dag a)b-1](aa^\dag + a^\dag a)=0$,
$(aa^\dag + a^\dag a)b-1 \in (aa^\dag + a^\dag a)_{-1}(0) = 0$.
Hence $(aa^\dag + a^\dag a)b=1$. Now, if $y \in \A$, then
$L_{aa^\dag + a^\dag a}(by)=y$. This proves the surjectivity of 
$L_{aa^\dag + a^\dag a}$.
\end{proof}

As a corollary of Propositions \ref{pro37} and \ref{pro38}, conditions that characterize  the invertibility of  $aa^\dag + a^\dag a$
will be given. 

\begin{theorem}\label{thm39}
Let $\A$ be a unital Banach algebra and consider $a \in \A^\dagger$.
Then, necessary and sufficient for $aa^\dag + a^\dag a\in \A^{-1}$
is that $aa^\dag + a^\dag a$ is regular, $(aa^\dag + a^\dag a)_{-1}(0) = 0$,
$a^\dag a \A \cap aa^\dag (1 - a^\dag a) \A= 0$ and $a^{-1}(0) \cap (a^\dag)^{-1}(0) = 0$.
\end{theorem}
\begin{proof}Apply Propositions \ref{pro37} and \ref{pro38}.
\end{proof}

\indent Next co-EP Banach space operators will be characterized. \par


\begin{proposition}\label{cor8}
Let $\X$ be a Banach space, $T\in \L(\X)$ be Moore-Penrose
invertible, and
$\lambda, \mu \in \mathbb{C} \setminus \{0 \}$.
The following statements are equivalent.\par
\begin{itemize}
\item[{\rm (i)}] $T$ is co-EP;

\item[{\rm (ii)}] $\lambda T + \mu T^{\dag}\in \L(\X)$ is invertible and
    $\R(T)\cap \R(T^{\dag})=0$;

\item[{\rm (iii)}] $\lambda T + \mu T^{\dag}\in \L(\X)$ is invertible and
there exists an idempotent $P\in \L(\X)$  such that $\R(T)\subseteq 
\R(P)$ and $\R(T^\dagger)\subseteq \N(P)$;

\item[{\rm (iv)}] $TT^{\dag}+T^{\dag}T\in \L(\X)$ is invertible
    and  $\R(T)\cap \R(T^{\dag})=0$;
    
\item[{\rm (v)}] $\lambda T + \mu T^{\dag}\in \L(\X)$ is invertible and
there exists an idempotent $Q\in \L(\X)$  such that $\R(I-Q)\subseteq 
\N(T)$ and $\R(Q)\subseteq \N(T^\dagger)$;

\item[{\rm (vi)}] $TT^{\dag}+T^{\dag}T\in \L(\X)$ is invertible
    and  $\R(I-TT^{\dag})\cap \R(I-T^{\dag}T)=0$;

\item[{\rm (vii)}] $\lambda T + \mu T^{\dag}\in \L(\X)$ is invertible and
$\R(I-TT^{\dag})\cap \R(I-T^{\dag}T)=0$.
\end{itemize}
\end{proposition}
\begin{proof}
Observe that under any of the hypotheses we have
$\L(\X)T+\L(\X)T^\dag = \L(\X)$ and
$T\L(\X)+T^\dag\L(\X) = \L(\X)$.
Note that according to \cite[Lemma 5.1]{KR},
$TT^{\dag}\L(\X)\cap T^{\dag}T\L(\X) = T\L(\X)\cap T^{\dag}\L(\X)=0$
if and only if $\R(T)\cap \R(T^{\dag})=0$.
Moreover, according to \cite[Lemma 5.1]{KR} and \cite[Lemma 2.1]{KR}, an equivalent condition
for $\L(\X)TT^{\dag} \cap \L(\X)T^{\dag}T =
\L(\X)T^{\dag}\cap \L(\X)T = 0$
is that $\R(I-TT^{\dag})\cap \R(I-T^{\dag}T)=0$.
To conclude the proof, apply Theorem \ref{thm7}.
\end{proof}


\indent To end this section, co-EP Banach space operators defined on finite dimensional
Banach spaces will be considered. \par

\begin{remark}\label{rem10}\rm Let $\X$ be a finite dimensional Banach space and consider
$T\in \L(\X)$. If $T$ is co-EP, then $\dim \X=2\dim \R(T)= 2 \dim \N(T)$.
In fact, since $T$ is Moore-Penrose invertible,
$\dim \N(T)+\dim \R(T^{\dag}) = \dim \X =
\dim \N(T^{\dag}) + \dim \R(T)$. In addition,
since $\dim \N(T)+\dim \R(T) = \dim \X$,
$\dim \N(T)=\dim \N(T^{\dag})$ and $\dim \R(T) =
\dim \R(T^{\dag})$.
Now well, according to \cite[Theorem 5.2]{KR}, $\dim \X=2\dim \R(T)= 2 \dim \N(T)$.
\end{remark}

\begin{proposition}\label{cor9}Let $\X$ be a finite dimensional Banach space and consider $T\in \L(\X)$
such that $T$ is Moore-Penrose invertibe.
The following statements are equivalent.\par
\begin{itemize}
\item[{\rm (i)}] $T$ is co-EP;
\item[{\rm (ii)}] $\R(T)\cap \R(T^{\dag})=0$ and
    $\N(T)\cap \N(T^{\dag})=0$;
\item[{\rm (iii)}] $ \X= \R(T)+ \R(T^{\dag})$ and
    $\X=  \N(T)+ \N(T^{\dag})$.
\end{itemize}
\end{proposition}

\begin{proof}
(i) $\Rightarrow$ (ii). According to Proposition \ref{cor8},
$\R(T)\cap \R(T^{\dag})=0$.
In addition, since $T-T^{\dag}\in \L(\X)$ is invertible,
$\N(T) \cap \N(T^{\dag})=0$.\par

\noindent (ii) $\Rightarrow$ (i).
Let $x\in \N(TT^{\dag}-T^{\dag}T)$, i.e., $TT^{\dag}(x)=T^{\dag}T(x)$.
Thus, $z=TT^{\dag}(x)=T^{\dag}T(x)\in \R(T)\cap \R(T^{\dag})=0$. As a result,
$x\in \N(T)\cap  \N(T^{\dag})=0$.\par

\noindent (i) $\Rightarrow$ (iii). Apply \cite[Theorem 5.2]{KR}.\par

\noindent (iii) $\Rightarrow$ (ii).
Since $T^{\dag}$ is Moore-Penrose invertible, 
$\dim \N(T)=\dim \N(T^{\dag})$ and $\dim \R(T)=\dim \R(T^{\dag})$.
Let $n=\dim \X$, $r=\dim \R(T)$, $s_1=\dim [\R(T) \cap \R(T^\dag)]$ and
$s_2=\dim [\N(T) \cap \N(T^\dag)]$.
According to the hyotheses,  $n = 2r-s_1$ and $n=2(n-r)-s_2$. Hence
$s_1+s_2=0$, which leads to $s_1=s_2=0$. 
\end{proof}

\section {\sfstp Hermitian co-EP elements}


\noindent Given a unital Banach algebra $\A$, if  $a\in \A$ is co-EP, then according to \cite[Theorem 3.2]{KR},
there exist  $h, k\in \A$, $h=h^2$ and $k=k^2$, such that
$h\A=a\A$, $(1-h)\A=a^{\dag}\A$, $\A k=\A a^{\dag}$ and $\A (1-k)=\A a$ (the idempotents $h$
and $k$ are unique).
Next a particular class of co-EP elements, the ones that their idempotents $h$ and $k$ are hermitian,
will be introduced.\par

\begin{definition} \label{def6} Let $\A$ be a unital Banach algebra and let $a\in \A$
such that $a$ is co-EP. Then, $a$ will be said hermitian co-EP if the above considered
idempotent $h$ is hermitian.
\end{definition}

To characterize hermitian co-EP Banach algebra elements some preparation is needed.\par

\begin{remark}\label{rema47}\rm  Let $\A$ be a unital Banach algebra and consider $a\in \A$ a Moore-Penrose
invertible element.\par
\noindent (i) Note that $a\in \A$ is hermitian  co-EP if and only if $a^\dagger$ is hermitian co-EP.\par
\noindent  (ii) Suppose that $a\in \A$ is  co-EP and consider the aforementioned idempotents
$h$, $k\in A$. In particular, 
$ha^\dag = 0$, $ha=a$, $a^\dag k = a^\dagger$ and $ak=0$.
Hence $h(aa^\dag - a^\dag a) = aa^\dag$ and 
$(aa^\dag - a^\dag a)k=aa^\dagger$, which implies that
$h=aa^\dagger(aa^\dag - a^\dag a)^{-1}$ and
$k=(aa^\dag - a^\dag a)^{-1}aa^\dagger$. 
\end{remark}

\indent In the following theorem hermitian co-EP Banach algebra elements will be characterized.
Compare with \cite[Theorem 2.9 and Corollary 2.10]{BR1}, \cite[Section 3]{BR2}
and \cite[Section 4]{BLR}. Note that given a unital Banach algebra $\A$, $a\in \A$ will be said to be \it bi-EP\rm,
if the idempotents $aa^\dag$ and $a^\dag a$ commute, see for example \cite{CM}.\par

\begin{theorem}\label{thm5} Let $\A$ be a unital Banach algebra and consider $a \in \A^\dag$.
Then, the following statements are equivalent.\par
\noindent {\rm (i)} $a$ is co-EP and $h$ is hermitian;\par
\noindent {\rm (ii)} $a$ is co-EP and $h=aa^{\dag}$;\par
\noindent {\rm (iii)} $a^{\dag}a +aa^{\dag}=1$;\par
\noindent {\rm (iv)} $a\A=a^{-1}(0)$;\par
\noindent {\rm (v)} $\A a=a_{-1}(0)$;\par
\noindent {\rm (vi)} $a$ is co-EP and $k=aa^{\dag}$;\par
\noindent {\rm (vii)} $a$ is co-EP and $k$ is hermitian;\par
\noindent {\rm (viii)} $a$ is co-EP and $k=h$;\par
\noindent {\rm (ix)} $a$ is co-EP and $a^2=0$;\par
\noindent {\rm (x) } $a^\dag \A = (a^\dag)^{-1}(0)$;\par
\noindent {\rm (xi) } $\A a^\dag = (a^\dag)_{-1}(0)$;\par
\noindent {\rm (xii) } $a$ is co-EP and bi-EP.

\end{theorem}
\begin{proof} {\rm (i) } $\Rightarrow$ {\rm (ii)}. Consider $L_h$, $L_{aa^{\dag}}\in 
\L(\A )$. Since $h\A=a\A=aa^{\dag}\A$,
$\R(L_h) = \R(L_{aa^{\dag}})$. 
Now well, since $h$ and $aa^{\dag}$ are hermitian, according to the
proof of \cite[Theorem 5 (ii)]{Bo}, $L_h$ and $L_{aa^{\dag}}$ 
are hermitian idempotents in $\L(\A )$.
Therefore, according to \cite[Theorem 2.2]{Pa}, $h=aa^{\dag}$.\par


\noindent {\rm (ii) } $\Rightarrow$ {\rm (iii)}. 
Since $(1-h)\A=a^{\dag}\A=a^{\dag}a\A$, 
$\R(L_{1-h}) = \R(L_{a^{\dag}a})$,
where $L_{1-h}$, $L_{a^{\dag}a}\in \L(\A )$. 
However, according to the hypothesis and to the proof  of \cite[Theorem 5 (ii)]{Bo},
$L_{1-h}$ and  $L_{a^{\dag}a}$ are hermitian idempotents in 
$\L(\A )$. Consequently, according to \cite[Theorem 2.2]{Pa},
$1-aa^{\dag}= 1-h=a^{\dag}a$.\par

\noindent {\rm (iii) } $\Rightarrow$ {\rm (iv)}. According to the hypothesis,
$$
a\A=aa^{\dag}\A= (1-a^{\dag}a)\A=(a^{\dag}a)^{-1}(0)=a^{-1}(0).
$$

\noindent {\rm (iv) } $\Rightarrow$ {\rm (v)}. Since $aa^{\dag}$ and $a^{\dag}a$ are idempotents and
$$
aa^{\dag}\A=a\A=a^{-1}(0)= (a^{\dag}a)^{-1}(0)=  (1-a^{\dag}a)\A,
$$
it is not dificult to prove that
$$
aa^{\dag}(1-a^{\dag}a) = 1-a^{\dag}a, \hskip.5truecm
(1-a^{\dag}a)aa^{\dag}=aa^{\dag},
$$
which in turn implies that
$$
a_{-1}(0)=(aa^{\dag})_{-1}(0) = 
(1-a^{\dag}a)_{-1}(0)=\A a^{\dag}a=\A a.
$$

\noindent {\rm (v) } $\Rightarrow$ {\rm (vi)}. 
Since $(a a^\dag - 1)a=0$, 
$a a^\dag - 1 \in a_{-1}(0) = \A a$. Hence, there
exists $u \in \A$ such that $ a a^\dag - 1 = ua$, and thus, 
$1 = (-u)a + aa^\dag \in \A a + \A a^\dagger$, which
proves $\A = \A a + \A a^\dagger$. Now it will be proved that  $\A a \cap \A a^\dagger = 0$. 
To this end, note that the hypothesis implies that $a^2=0$. If 
$x \in \A a \cap \A a^\dagger$, then there exist
$y,z \in \A$ such that $x=ya=za^\dagger$. So, $0=ya^2a^\dag = za^\dag a a^\dag = za^\dag = x$.
Thus, $ \A = \A a \oplus \A a^\dagger$.

Using an argument similar to the one in the proof of (iv) $\Rightarrow$ (v),
it is not difficult to prove that  $\A a^\dag a = \A(1 - aa^\dag)$, which
leads to $a^\dag a (1 - aa^\dag) = a^\dag a$ and
$(1 - aa^\dag) a^\dag a = 1 - aa^\dag$.
In particular, $a^{-1}(0) = a\A$. As in the previous paragraph it is possible to prove
that $\A = a\A \oplus a^\dag \A$. By Theorem \ref{thm7}, $a$ is co-EP.	

Since $\A aa^\dag = \A a^{\dag}$, $\A (1-aa^\dag) = a_{-1}(0) = \A a$ and
the idempotent $k$ is unique, $k=aa^{\dag}$.\par

\noindent {\rm (vi) } $\Rightarrow$ {\rm (vii)}. Clear

\noindent {\rm (vii) } $\Rightarrow$ {\rm (viii)}. Since $\A k=\A a^{\dag}=\A aa^{\dag}$,
an argument similar to the one used to prove that statement (i) implies statement (ii)
proves that $k=aa^{\dag}$. In addition, since $k$ is hermitian and 
$\A (1-k)=\A a=\A a^{\dag}a$,
according \cite[Theorem 2.2]{Pa}, 
$1-aa^{\dag}=1-k=a^{\dag}a$, equivalently,
$aa^{\dag}+a^{\dag}a=1$. 
Since $aa^\dag \A = a\A$, $(1 - aa^\dag)\A = a^\dag a \A = a^\dag \A$
and the idempotent  $h$ is unique,  $h=aa^\dagger = k$.
\par

\noindent {\rm (viii) } $\Rightarrow$ {\rm (ix)}.
Since $a \in a\A = h\A = k\A$ and $a \in \A a = \A(1-k)$, there exist
$u,v \in \A$ such that $a=ku$ and $a=v(1-k)$. 
Hence  $a^2 = v(1-k)ku=0$.\par

\noindent {\rm (ix) } $\Rightarrow$ {\rm (iv)}. 
Since $a^2=0$,  $a\A \subseteq a^{-1}(0)$. Let
$x \in a^{-1}(0)$. Since
$a$ is co-EP, from Theorem \ref{thm7}, such $x$ can be written as
$x = au+a^\dag v$. Thus, $0=ax=a^2u+aa^\dag v = aa^\dag v$, hence
$0=a^\dag v$. So, $x=au \in a\A$.\par

\noindent {\rm (iv) } $\Rightarrow$ {\rm (i)}. 
Since (iv) $\Rightarrow$ (vii) $\Rightarrow$ (viii),  $h$ is hermitian. \par

\noindent {\rm (i) } $\Leftrightarrow {\rm (x) }$ (respectively {\rm (xi)}). Since $a$ is hermitian  co-EP if and only if $a^\dagger$ is hermitian co-EP,
the equivalence between statments (i) and (x) (respectively between statments (i) and (xi))
can be derived applying the Theorem to $a^\dagger$  using the equivalence between
statments (i) and (iv) (respectively between statments (i) and (v)).\par

\noindent {\rm (viii) } $\Leftrightarrow {\rm (xii)}$ Note that according to Remark \ref{rema47} (ii),
\begin{eqnarray*}
h =k & \iff &
aa^\dagger(aa^\dag - a^\dag a)^{-1} = (aa^\dag - a^\dag a)^{-1} aa^\dagger \\ 
& \iff &
(aa^\dag - a^\dag a)aa^\dag = aa^\dag (aa^\dag - a^\dag a) \\
& \iff & 
(aa^\dag)(a^\dag a) = (a^\dag a)(aa^\dag).
\end{eqnarray*}\end{proof}

\begin{remark}
{\rm 
Let $\A$ be a unital Banach algebra. 
Recall that an element $a\in \A^{\dag}$ is said EP when $aa^\dag - a^\dag a = 0$ and 
$a$ is said co-EP when $aa^\dag - a^\dag a \in \mathcal{A}^{-1}$. What is more,
according to Theorem \ref{thm5},  $a\in \A$ is hermitian co-EP if  $a^{\dag}a +aa^{\dag}=1$. However,  the situation
$aa^\dag - a^\dag a = 1$ is impossible. 
In fact, assume $aa^\dag - a^\dag a = 1$. By postmultiplying it by $a$ one gets
$a^\dag a^2 = 0$, hence $a^2=0$. By premultiplying $aa^\dag - a^\dag a = 1$ by $a$
one has $a=0$, which is unfeasible in view of $aa^\dag -a^\dag a=1$.
}
\end{remark}


\section {\sfstp  Condition  ($\mathcal{ P}$) and  the perturbations of  co-EP elements}


In this section the condition  ($\mathcal{ P}$) will be introduced and
the perturbations of  co-EP elements  will be studied. The condition ($\mathcal{ P}$)
resembles the condition ($\mathcal{ W}$) studied for perturbations of
Drazin invertible elements (\cite{VY}, \cite{YW}).

\indent Let $a\in \mathcal{ A}$ be Moore-Penrose invertible. The element  $b\in
\mathcal{ A}$ will be said that {\it obeys the condition} ($\mathcal{ P}$) {\it at} $a$ if
$$
b - a = aa^\dag(b - a)a^\dag a\quad \text {and}\quad \|a^\dag(b -
a)\|< 1.$$

\noindent Note  that the condition
$$
b - a = aa^\dag(b - a)a^\dag a\
$$
is equivalent to 
\begin{equation}\label{1.4}
b - a = aa^\dag(b - a)= (b - a)a^\dag a .
\end{equation}
Basic auxiliary results are summarized in the following lemma.


\bigskip
\begin{lemma}\label{lem1.2}
 Let $\A$ be a unital Banach algebra and consider $a \in \mathcal{ A}$, a Moore-Penrose invertible element.
If $b\in \mathcal{ A}$ obeys the condition $\mathcal{ P}$ at $a$. Then

\item {\rm (i)}\quad $b = a(1 + a^\dag(b - a))$;

\item {\rm (ii)}\quad $b = (1 + (b - a)a^\dag)a$;

\item {\rm (iii)}\quad $1 + a^\dag(b - a)$ and
$1 + (b - a)a^\dag$ are invertible, and

\begin{equation}\label{1.5}
(1 + a^\dag(b - a))^{-1}a^\dag = 
a^\dag(1 + (b - a)a^\dag)^{-1}.
\end{equation}
\end{lemma}
\begin{proof}
To prove (i) and (ii) observe that by \eqref{1.4}
$$
b = a + (b - a) = a + aa^\dag(b - a) = a(1 + a^\dag(b - a))
$$
and
$$
b = a + (b - a) = a + (b - a)a^\dag a = (1 + (b - a)a^\dag)a.
$$
Clearly, the condition  $\|a^\dag(b - a)\| < 1$ implies  that 
$1 + a^\dag (b - a)$ is invertible, and so  
$1 + (b - a)a^\dag$ is
invertible. Finally, \eqref{1.5} follows by direct verification.
\end{proof}

Now  the main results of this section  will be proved.


\begin{theorem}\label{thm2.1} 
Let $\A$ be a unital Banach algebra and consider $a \in \mathcal{ A}$, 
a Moore-Penrose invertible element. 
If $b\in \mathcal{ A}$ obeys the condition $\mathcal{ P}$ at
$a$, then
 $b$ is Moore-Penrose invertible, $bb^\dag = aa^\dag$, $b^\dag  b=
 a^\dag a$ and
$b^\dag = (1 + a^\dag(b - a))^{-1}a^\dag
 = a^\dag(1 + (b - a)a^\dag)^{-1}$.
\end{theorem}
\begin{proof}
By Lemma \ref{lem1.2} (iii), $1 + a^\dag
(b - a)$ and $1 + (b - a)a^\dag$ are invertible and
$$
(1 + a^\dag(b - a))^{-1}a^\dag = 
a^\dag(1 + (b - a)a^\dag)^{-1}.
$$
Set $\tilde b = (1 + a^\dag(b - a))^{-1}a^\dag = 
a^\dag(1 + (b - a)a^\dag)^{-1}$. 
Then, $b$ is Moore-Penrose invertible and
that $b^\dag = \tilde b$.  In fact, by Lemma \ref{lem1.2} (i),
\begin{equation}\label{2.2}
b\tilde b =  a(1 + a^\dag(b - a))
(1 + a^\dag(b - a))^{-1}a^\dag =
aa^\dag\end{equation} and by Lemma \ref{lem1.2} (ii),
\begin{equation}\label{2.3} 
\tilde bb =  a^\dag(1 + (b - a)a^\dag)^{-1}
(1 + (b - a)a^\dag)a = a^\dag a.
\end{equation}
 Hence, by \eqref{2.2} and Lemma \ref{lem1.2} (i),
\begin{equation}\label{2.4}
b\tilde b b= (aa^\dag)b = 
(aa^\dag a)(1 + a^\dag(b - a)) = 
a(1 + a^\dag(b - a))=b.
\end{equation} 
Furthermore,
 by \eqref{2.3},
\begin{equation}\label{2.5p}
\tilde b b\tilde b= (a^\dag a)a^\dag(1 + (b - a)a^\dag)^{-1} =
(a^\dag  a a^\dag)(1 + (b - a)a^\dag)^{-1} = 
a^\dag (1 + (b - a)a^\dag)^{-1} = \tilde b.
\end{equation}
 Finally, by \eqref{2.2}, \eqref{2.3}, \eqref{2.4} and \eqref{2.5p},
 $b$ is Moore-Penrose invertible with
 $b^\dag = \tilde b$.
\end{proof}

The following corollary consists in a direct application of Theorem \ref{thm2.1}.\par
\bigskip
\begin{corollary}\label{ 2.1} Let $\A$ be a unital Banach algebra and consider $a \in \mathcal A^{EP}_{co}$. If $b\in
\mathcal{ A}$ obeys the condition $\mathcal{ P}$ at $a$, then $b\in \mathcal
A^{EP}_{co}$, $bb^\dag = aa^\dag$ and $b^\dag b = a^\dag a$.
\end{corollary}


\begin{corollary}\label{ 2.2} Let $\A$ be a unital Banach algebra and consider $a \in \mathcal{ A}$, a Moore-Penrose invertible element.
If $b\in \mathcal{ A}$ obeys the condition $\mathcal{ P}$ at $a$, then
 $b$ is Moore-Penrose invertible and 
\begin{equation}\label{2.5}
\dfrac {\|b^\dag - a^\dag\|}{\|a^\dag\|} \le \dfrac {\|a^\dag(b -
a)\|} {1 - \|a^\dag(b - a)\|}.
\end{equation}
 \end{corollary}
\begin{proof}
By Theorem \ref{thm2.1} 
\begin{align*}
b^\dag - a^\dag &= (1 + a^\dag(b - a))^{-1}a^\dag - a^\dag\\
&= [(1 + a^\dag(b - a))^{-1} - 1]a^\dag\\
 &=- (1 + a^\dag(b - a))^{-1}(a^\dag(b - a))a^\dag.\\
\end{align*}
Hence
\begin{align*}
\|b^\dag - a^\dag\| &\le \|(1 + a^\dag(b - a))^{-1}\| 
\|a^\dag(b - a)\| \|a^\dag\|\\
&\le \dfrac {\|a^\dag(b - a)\|} {1 - \|a^\dag(b - a)\|} \|a^\dag\|,
\end{align*}
and \eqref{2.5} holds. 
\end{proof}

\begin{corollary}\label{cor 2.3} Let $\A$ be a unital Banach algebra and consider $a \in \mathcal{ A}$, a Moore-Penrose invertible element.
If  $b\in \mathcal{ A}$ obeys the condition $\mathcal{ P}$ at $a$, then
 $b\in \A$ is Moore-Penrose invertible and
\begin{eqnarray}\label{2.8}
\dfrac {\|a^\dag\|} {1 +  \|a^\dag(b - a)\|} \le \|b^\dag\| \le
\dfrac {\|a^\dag\|} {1 - \|a^\dag(b - a)\|}.
\end{eqnarray}
 \end{corollary}
\begin{proof}
By Theorem \ref{thm2.1},
\begin{eqnarray*} 
b^\dag =   (1 + a^\dag(b - a))^{-1}a^\dag, 
\end{eqnarray*}
hence
\begin{eqnarray*} 
a^\dag =   (1 + a^\dag(b - a))b^\dag.
\end{eqnarray*}
Now
\begin{eqnarray*}\label{2.11}
\|b^\dag\| \le  \|(1 + a^\dag(b - a))^{-1}\| \|a^\dag\| \le
 \dfrac {\|a^\dag\|}
{1 - \|a^\dag(b - a)\|},
\end{eqnarray*}
\begin{eqnarray*} 
\|a^\dag\| \le  \|(1 + a^\dag(b - a))\| \|b^\dag\| \le 
(1 + \|a^\dag(b - a)\| )\|b^\dag\|, 
\end{eqnarray*}
which proves \eqref{2.8}.
\end{proof}

\begin{corollary}\label{cor 2.4} 
Let $\A$ be a unital Banach algebra and consider $a \in \mathcal{ A}$, a Moore-Penrose invertible element.
If $b\in \mathcal{ A}$ obeys the condition $\mathcal{ P}$ at $a$ and $\|a^\dag (b
- a)\|< 1/2$,
then $b$ is Moore-Penrose invertible and
  $a$
obeys the condition $\mathcal{ P}$ at $b$.
 \end{corollary}
\begin{proof}
By Theorem \ref{thm2.1}, $bb^\dag =
aa^\dag$ and $b^\dag b= a^\dag a$. Hence $a - b = bb^\dag(a -
b)bb^\dag$. Again, by Theorem \ref{thm2.1}, 
$$\aligned
\|b^\dag(a - b)\| &= \|(1 + a^\dag(b - a))^{-1}a^\dag (a - b)\|\\
&\le \|(1 + a^\dag(b - a))^{-1}\|\, \| a^\dag(a - b)\|\\
& \le \dfrac { \| a^\dag(a - b)\|} {1 - \|a^\dag(b - a)\|}< 1.
\endaligned
$$
This completes the proof.   
\end{proof}


\begin{corollary}\label{col 2.5} Let $\A$ be a unital Banach algebra and consider $a \in \mathcal{ A}$, a 
Moore-Penrose invertible element. If $b\in \mathcal{ A}$ obeys  the condition $\mathcal{ P}$ at $a$
and  $\|a^\dag\| \|b - a\| < 1$, then
 $b$ is Moore-Penrose invertible and 
\begin{eqnarray*} 
\dfrac {\|b^\dag - a^\dag\|}{\|a^\dag\|} \le \dfrac {k_\dag(a) \|b -
a\| / \|a\|} {1 - k_\dag(a) \|(b - a)\| / \|a\|}, 
\end{eqnarray*}
where
$$
k_\dag(a) = \|a\| \|a^\dag\|
$$
is defined as the condition number with respect to the Moore-Penrose
inverse.
 \end{corollary}
\bigskip

To finish this section, some algebraic properties of
the Moore-Penrose inverse will be considered. Recall that  if
$a, b\in \mathcal{ A}$ are Moore-Penrose invertible, then $ab$ is not
necessary Moore-Penrose invertible, and if $ab$ is Moore-Penrose
invertible, then in general $(ab)^\dag \ne b^\dag a^\dag$. In the
next proposition, it will be shown that if $b$ obeys the condition $\mathcal{ P}$
at $a$, then it is possible to give  precise characterizations.

\begin{proposition}\label{pro 2.6} Let $\A$ be a unital Banach algebra and consider $a \in \mathcal{ A}$, a 
Moore-Penrose invertible element. If  $b\in \mathcal{ P}$ obey  the condition $\mathcal{ P}$ at
$a$, then the following statements hold.

\item {\rm (i)}\quad $ab^\dag$ has Moore-Penrose inverse and
$(ab^\dag)^\dag = ba^\dag$; moreover $a\in \mathcal{ A}_{co}^{EP}$ implies
$ab^\dag\in \mathcal{ A}_{co}^{EP}$.

\item {\rm (ii)}\quad $b^\dag a$ has Moore-Penrose inverse and
$(b^\dag a)^\dag = a^\dag b$; moreover $a\in \mathcal{ A}_{co}^{EP}$
implies $b^\dag a\in \mathcal{ A}_{co}^{EP}$.

\item \rm {(iii)}\quad $ba^\dag$ has Moore-Penrose inverse and
$(ba^\dag)^\dag = ab^\dag$; moreover $a\in \mathcal{ A}_{co}^{EP}$ implies
$ba^\dag\in \mathcal{ A}_{co}^{EP}$.

\item {\rm (iv)}\quad $a^\dag b$ has Moore-Penrose inverse and
$(a^\dag b)^\dag = b^\dag a$; moreover $a\in \mathcal{ A}_{co}^{EP}$
implies $a^\dag b\in \mathcal{ A}_{co}^{EP}$.
\end{proposition}
\begin{proof} 
It is enough  to prove (i) and (ii). To prove (i) recall that
 by Theorem \ref{thm2.1}  $bb^\dag = aa^\dag$.
Hence
\begin{eqnarray}\label{2.16x}
(ab^\dag)(ba^\dag) = a(b^\dag b)a^\dag= a(a^\dag a)a^\dag = aa^\dag
\end{eqnarray}
and
\begin{eqnarray}\label{2.16xx}
(ba^\dag)(ab^\dag) = b(a^\dag a)b^\dag= b(b^\dag b)b^\dag = bb^\dag
= aa^\dag.
\end{eqnarray}

Further,
\begin{eqnarray}\label{2.16}
(ab^\dag)(ba^\dag)(ab^\dag) = aa^\dag(ab^\dag)=  (aa^\dag a)b^\dag = ab^\dag.
\end{eqnarray}
 Finally,
\begin{eqnarray}\label{2.17}
(ba^\dag)(ab^\dag)(ba^\dag) = (ba^\dag)aa^\dag=  b(a^\dag aa^\dag)=
ba^\dag.
\end{eqnarray} Now by \eqref{2.16x}, \eqref{2.16xx}, \eqref{2.16} and \eqref{2.17}, (i) holds. 
Statement (ii) can be proved similarly.
\end{proof}

\vskip.3truecm
\noindent Julio Ben\'{\i}tez\par
\noindent E-mail address: jbenitez@mat.upv.es\par
\vskip.3truecm
\noindent Enrico Boasso\par
\noindent E-mail address: enrico\_odisseo@yahoo.it \par
\vskip.3truecm
\noindent Vladimir Rako\v cevi\'c\par
\noindent E-mail address: vrakoc@ptt.rs\par
\end{document}